\documentclass[smallextended,numbook,runningheads]{svjour3}     % onecolumn (second format)
\smartqed  % flush right qed marks, e.g. at end of proof
\usepackage{graphicx}
\usepackage{amsmath}
\usepackage{epstopdf} % needed if you have eps figures in a pdflatex manuscript
\usepackage{mathptmx}      % use Times fonts if available on your TeX system
\usepackage{t1enc,latexsym,amssymb,verbatim,epsfig}
\journalname{Journal}
\begin{document}

\title{On the zero-stability of multistep methods on smooth nonuniform grids}

%\titlerunning{Zero-stability of multistep methods on smooth nonuniform grids}        % if too long for running head

\author{Gustaf S\"oderlind         \and
        Imre Fekete         \and
        Istv\'{a}n Farag\'o
}


\institute{Gustaf S\"oderlind, \email{Gustaf.Soderlind@na.lu.se,Gustaf.Soderlind@gmail.com} \at
           Centre for Mathematical Sciences, Lund University,
           Box 118, SE-221 00 Lund, Sweden
           %  \\
           \and
           Imre Fekete,  \email{feipaat@cs.elte.hu} \at
           Department of Applied Analysis and Computational Mathematics, E\"{o}tv\"{o}s Lor\'{a}nd University, P\'{a}zm\'{a}ny P. s. 1/C, H-1117 Budapest, Hungary. MTA-ELTE Numerical Analysis and Large Networks Research Group,  P\'{a}zm\'{a}ny P. s. 1/C, H-1117 Budapest, Hungary. Supported by the \'UNKP-17-4 New National Excellence Program of the Ministry of Human Capacities
           \and
           Istv\'{a}n Farag\'{o},  \email{faragois@cs.elte.hu} \at
           Department of Applied Analysis and Computational Mathematics, E\"{o}tv\"{o}s Lor\'{a}nd University, P\'{a}zm\'{a}ny P. s. 1/C, H-1117 Budapest, Hungary. MTA-ELTE Numerical Analysis and Large Networks Research Group,  P\'{a}zm\'{a}ny P. s. 1/C, H-1117 Budapest, Hungary. Supported by the Hungarian Scientific Research Fund OTKA under grants No.\ 112157 and 125119.
           }

\date{Received: date / Accepted: date}
% THIS IS THE REVISED VERISON ADDRESSING ISSUES RAISED BY REFEREES
% THIS IS THE NEW REVISED VERISON RE-ADDRESSING ISSUES RAISED BY REFEREES

\maketitle

\begin{abstract}
In order to be convergent, linear multistep methods must be zero stable. While constant step size theory was established in the 1950's, zero stability on nonuniform grids is less well understood. Here we investigate zero stability on compact intervals and smooth nonuniform grids. In practical computations, step size control can be implemented using smooth (small) step size changes. The resulting grid $\{t_n\}_{n=0}^N$ can be modeled as the image of an equidistant grid under a smooth deformation map, i.e., $t_n = \Phi(\tau_n)$, where $\tau_n = n/N$ and the map $\Phi$ is monotonically increasing with $\Phi(0)=0$ and $\Phi(1)=1$. The model is justified for any fixed order method operating in its asymptotic regime when applied to smooth problems, since the step size is then determined by the (smooth) principal error function which determines $\Phi$, and a tolerance requirement which determines $N$. Given any strongly stable multistep method, there is an $N^*$ such that the method is zero stable for $N>N^*$, provided that $\Phi \in C^2[0,1]$. Thus zero stability holds on all nonuniform grids such that adjacent step sizes satisfy $h_n/h_{n-1} = 1 + \mathrm O(N^{-1})$ as $N\rightarrow\infty$. The results are exemplified for BDF-type methods.
%% ===== REVISED ABSTRACT 2018-03-19/20 CLARIFYING THAT THE MAP \Phi MODELS STEP SIZE VARIATION ===== %%

%Include keywords and mathematical subject classification numbers as needed.
\end{abstract}

\section{Introduction}\label{introduction}

A linear multistep method, discretizing an initial value problem $\dot y = f(t,y)$, is represented by a difference equation of order $k$,
\begin{equation}\label{LMM}
{\frac {1}{h}}\sum_{j=0}^k \alpha_{j}y_{n+j} = \sum_{j=0}^k \beta_{j}f(t_{n+j},y_{n+j}).
\end{equation}
Here the step size $h = t_{n+k}-t_{n+k-1} > 0$ is assumed constant. We denote the forward shift operator by $\mathrm E$ and write the method $h^{-1}\rho(\mathrm E)y_n = \sigma(\mathrm E)f(t_n,y_n)$, with generating polynomials
\begin{equation}\label{rhooperators}
\rho(\zeta) = \sum_{j=0}^k \alpha_{j}\zeta^j = (\zeta - 1) \sum_{j=0}^{k-1} \gamma_{j}\zeta^j = (\zeta-1)\cdot \rho_R(\zeta),\qquad \sigma(\zeta) = \sum_{j=0}^k \beta_{j}\zeta^j.
\end{equation}
These are arranged to have no common factors, and coefficients are normalized by $\sigma(1)=1$. \textit{Zero stability} is necessary for convergence, and requires that all roots of $\rho(\zeta)=0$ lie inside or on the unit circle, with no multiple unimodular roots. Since consistent methods have $\rho(\zeta) = (\zeta-1)\cdot \rho_R(\zeta)$ as indicated above, zero stability is a condition on the \textit{extraneous operator} $\rho_R(\zeta)$. Its zeros are referred to as the extraneous roots. Strong zero stability requires that \textit{all extraneous roots are strictly inside the unit circle}; this is a condition on the $k$ coefficients $\{\gamma_j\}_{j=0}^{k-1}$.

Since the extraneous operator is void in Adams-Moulton and Adams-Basforth methods, these methods are trivially zero stable for variable steps, \cite[p.\ 407]{HairNW1993}. The most important case having a nontrivial extraneous operator is the BDF methods, known to be zero stable for $1\leq k\leq 6$, cf.\ \cite{Cryer1972}, \cite[p.\ 381]{HairNW1993}. Some (nonstiff) method suites, such as the dcBDF and IDC methods \cite{ArevFS2000}, are based on the BDF $\rho$ operator, and have the same zero stability properties for $k\geq 2$. Other examples of nontrivial extraneous operators are the weakly stable explicit midpoint method (two-step method of order 2) and the lesser used weakly stable implicit Milne methods, \cite[p.\ 363]{HairNW1993}.

Adaptive computations are of particular importance for stiff problems, as widely varying time scales call for correspondingly large variations in step size. Of the methods mentioned above, only the BDF family has unbounded stability regions specifically designed for stiff problems. Thus the BDF methods must handle step size variations well, in spite of its extraneous operator, explaining why studies of variable step size zero stability mostly center on the BDF methods, \cite[p.\ 402ff]{HairNW1993}.

Although there are several ways to construct multistep methods on nonuniform grids, we shall only consider the \emph{grid-independent representation} of multistep methods, \cite{ArevSo2016}. This represents a multistep method on any nonuniform grid using a fixed parametrization, defining a computational process where the coefficients $\alpha_{j,n}, \beta_{j,n}$ vary along the solution and depend on $k-1$ consecutive step size ratios. For simplicity, but without loss of generality, let us consider a quadrature problem $\dot y = f(t)$ on $[0,1]$ using variable steps. The multistep method (\ref{LMM}) becomes
\begin{equation}\label{LMMvar}
{\frac {1}{h_{n+k-1}}}\sum_{j=0}^k \alpha_{j,n}y_{n+j} = \sum_{j=0}^k \beta_{j,n}f(t_{n+j}),
\end{equation}
where $h_{n+k-1} = t_{n+k}-t_{n+k-1}$. Letting $y\in C^{p+1}$ denote the exact solution, we obtain
\begin{equation}\label{LMMvarex}
{\frac {1}{h_{n+k-1}}}\sum_{j=0}^k \alpha_{j,n}y(t_{n+j}) = \sum_{j=0}^k \beta_{j,n}f(t_{n+j}) - c_n h^{p}_{n+k-1} y^{(p+1)}(\vartheta),
\end{equation}
provided that $y$ is sufficiently differentiable, and where $\vartheta\in[t_n,t_{n+k}]$. Subtracting (\ref{LMMvarex}) from (\ref{LMMvar}) gives
\begin{equation}\label{LMMvarerror}
{\frac {1}{h_{n+k-1}}} \sum_{j=0}^k \alpha_{j,n}e_{n+j} = c_n h^{p}_{n+k-1} y^{(p+1)}(\vartheta),
\end{equation}
where the global error at $t_n$ is $e_{n} = y_n - y(t_n)$. Here, the local error $c_n h^{p}_{n+k-1} y^{(p+1)}(\vartheta)$ goes to zero if $h_{n+k-1}\rightarrow 0$ (consistency), but convergence ($e_n\rightarrow 0$) in addition requires that solutions to the homogeneous problem
\begin{equation}\label{LMMhom}
{\frac {1}{h_{n+k-1}}} \sum_{j=0}^k \alpha_{j,n}e_{n+j} = 0
\end{equation}
remain bounded. Thus zero stability on nonuniform grids is investigated in terms of the problem $\dot y = 0$ and finding sufficient conditions on the grid partitioning of $[0,1]$, such that the numerical solution $\{y_n\}_0^N$ is uniformly bounded as $N\rightarrow\infty$. This problem has been approached in several different ways, see e.g.\ \cite{ButchHeard2002}, \cite{CrouLi1984}, \cite{Grigor1983}, \cite{GuglZe2001}, usually with the aim of finding precise bounds on the step size ratios, such that the method remains convergent. Since the method coefficients change from step to step, most analyses become highly complicated. For example, the problem can be addressed by studying infinite products of companion matrices associated with the recursion (\ref{LMMhom}), \cite[p.\ 403]{HairNW1993}, or by considering the nonuniform grid as a ``perturbation'' of an equidistant grid, by letting the step size vary smoothly, \cite{GearTu1974}.

%% ===== THE FOLLOWING PARAGRAPHS ARE COMPLETELY REWORKED TO CLARIFY ISSUES RAISED BY ONE OF THE REFEREES ===== %%
An overview is given in \cite[p.\ 402ff]{HairNW1993}, but the classical results focus on the existence of local step size ratio bounds that guarantee zero stability. By constrast, our focus is on grid smoothness. Using (near) Toeplitz operators, our aim is to develop a proof methodology for adaptive computation, aligned with the formal convergence analysis in the Lax--Stetter framework, cf.\ \cite{Stette1973}. We let the grid points be given by a strictly increasing sequence $\{t_n\}_0^N$ and define the step sizes by $h_n = t_{n+1}-t_n$, requiring that $h_n\rightarrow 0$ for every $n$ as $N\rightarrow\infty$. If the grid is smooth enough, then any multistep method which is strongly zero stable on a uniform grid is also zero stable on the nonuniform grid for $N$ large enough.

%% ===== REVISED ONCE MORE 2018-03-19 ===== %%
The main result has the following structure. Every multistep method is associated with two constants, $C_0$ and $C_k$, where the former only depends on constant step size theory, and is bounded if the method is strongly zero stable on a uniform grid. The second constant depends on the first order variation of the method's coefficients for infinitesimal step size variations, and is computable using a suitable computer algebra system such as \textsc{Maple} or \textsc{Mathematica}. Finally, grid smoothness will be characterized in terms of a differentiable grid deformation map, requiring a bound on a function of the form $\varphi'/\varphi$. Under these conditions, the method is zero stable on the non-uniform grid provided that
$$
{\frac {\max |\varphi'/\varphi|}{N}}\cdot C_0\cdot C_k < 1.
$$
This separates method properties and grid properties, and only requires that the total number of steps $N$ is large enough. The important issues are to generate a smooth step size sequence (which automatically manages step size ratios), and using a sufficiently small error tolerance, which determines $N$. Although such step size sequences can easily be constructed in adaptive computation, \cite{Soderl2003}, most multistep codes still use comparatively large step size changes, violating the smoothness conditions required for zero stability. This has been demonstrated to be a likely cause of poor computational stability observed in practice, \cite{SoderWang2006}. In production codes it is often thought that frequent, small step size changes are not ``worthwhile,'' but the present paper and classical theory only support such step size changes.

Our approach is intended as an analysis tool for deriving a rigorous convergence proof for adaptive multistep methods, redefining practical implementation principles. A full convergence analysis of the initial value problem $\dot y = f(t,y)$ requires further attention to detail, as it also involves the Lipschitz continuity of the vector field $f$ with respect to $y$, as well as (for implicit methods) the solvability of equations of the form $v = \gamma\, hf(t,v) + w$. The solvability will depend on the magnitude of the Lipschitz constant $L[\gamma\,hf]$ or the logarithmic Lipshitz constant $M[\gamma\,hf]$, see e.g.\ \cite{Soderl2006}. Likewise, error bounds will depend on these quantities. Here, however, we only focus on zero stability, which can be fully characterized in the simpler setting of a quadrature problem. We shall return to the full convergence analysis on smooth nonuniform grids in a forthcoming study.

\section{Smooth nonuniform grids}

If an initial value problem has a smooth solution, then the step size sequence, keeping the local error (nearly) constant, is also smooth, \cite{GearTu1974}, \cite{Shampi1985}. A smooth sequence is also known to be necessary in connection with e.g.\ Hamiltonian problems, \cite{HairSo2005}, as well as in finite difference methods for boundary value problems. For these reasons, we shall model nonuniform grids by a smooth deformation of an equidistant grid. We only consider compact intervals.\medskip

\noindent
\textbf{Adaptive computation.} The asymptotic behavior of the local error per unit step in a multistep method is of the form $l_n = ch_n^{p}y^{(p+1)}$. The most common step size control in adaptive computation aims to keep $\|l_n\|$ constant, equal to a given local error tolerance $\epsilon$. Representing the step size in terms of a \emph{step size modulation function} $\mu(t)$ allows the step size at time $t$ to be expressed as $h = \mu(t)/N$, so that the ``ideal'' step size sequence can be modeled by 
$$
c \left( {\frac {\mu(t_n)}{N}} \right)^{p} \|y^{(p+1)}(t_n)\| = \epsilon.
$$
It follows that $N \sim \epsilon^{-1/p}$. In other words, \textit{the local error tolerance determines} $N$. By contrast, $\mu(t)$ is \textit{determined by the problem}. It is smooth if $y^{(p+1)}(t)$ is smooth, since
$$
\mu(t) \sim \|y^{(p+1)}(t)\|^{-1/p}.
$$
In real adaptive computations, a step size control of the form $h_n = r_{n-1}h_{n-1}$ is used, where the step ratio sequence is processed by a digital filter to generate a smooth step size sequence, \cite{Soderl2003}. This may e.g.\ take the form
$$
r_{n-1} = \left( {\frac {\epsilon}{\|l_{n-1}\|}}\right)^{b_1/p} \left( {\frac {\epsilon}{\|l_{n-2}\|}} \right)^{b_2/p} r_{n-2}^{-a_1},
$$
where $l_{n-1}$ and $l_{n-2}$ are local error estimates and $(b_1,b_2,a_1)$ are the filter parameters. The controller keeps the local error close to the tolerance $\epsilon$. As a consequence the step ratios will remain near $1$. Further, reducing the tolerance $\epsilon$ increases $N$, reducing step sizes as well as \emph{step ratios}. Thus it is justified to model a nonuniform grid by a smooth grid deformation, and such a grid can be generated using a proper filter to continually adjust the step size. It also corresponds well to the behavior observed in computational practice when such step size controllers are employed.%% REWORKED TO CONVINCE REFEREE THAT THE STEP SIZE SEQUENCE MODEL IS JUSTIFIED 2018-03-19
\medskip

\noindent
\textbf{Modeling a smooth nonuniform grid.} Let $\Phi:\tau\mapsto t$ be a smooth, strictly increasing map in $C^2[0,1]$, satisfying $\Phi(0)=0$ and $\Phi(1)=1$. Further, let its derivative $\Phi' = \mathrm d\Phi/\mathrm d\tau$ be denoted by $\varphi$ and assume that $\varphi'/\varphi \in L^{\infty}[0,1]$. Now, given $N$, let $\tau_n = n/N$ and construct a smooth nonuniform grid $\{t_n\}_{n=0}^N$ by
\begin{equation}\label{gridmap}
t_n = \Phi(\tau_n).
\end{equation}
Since $t = \Phi(\tau)$ we have the differential relation $\mathrm d\,t = \varphi(\tau)\,\mathrm d\tau$. By a discrete correspondence, mesh widths are related by $\Delta t \approx \varphi(\tau)\, \Delta\tau$. Thus we model the  \emph{step size sequence} $\{h_n\}_{n=0}^{N-1}$ by
\begin{equation}\label{stepseq}
h_n = t_{n+1} - t_n = \Phi(\tau_{n+1})-\Phi(\tau_{n}) \approx \varphi(\tau_{n+1/2})/N.
\end{equation}
Hence $h_n \rightarrow 0$ as $N\rightarrow\infty$. This allows us to study zero stability on nonuniform grids in terms of the single-parameter limit $N\rightarrow\infty$. This does not substantially restrict $h_{\max}/h_{\min}$ during the overall integration, although adjacent step ratios will be small.\medskip

\noindent
\textbf{Step ratios.} The coefficients of a multistep method on a nonuniform grid depend on the ratio of adjacent step sizes. By (\ref{stepseq}) the \emph{step ratios} $\{r_n\}_{n=0}^{N-2}$ are given by
\begin{equation}\label{stepratio}
r_{n-1} = {\frac {h_n}{h_{n-1}}} \approx {\frac {\varphi(\tau_{n+1/2})}{\varphi(\tau_{n-1/2})}} \approx {\frac {\varphi(\tau_n) + \varphi'(\tau_n)/(2N)}{\varphi(\tau_n) - \varphi'(\tau_n)/(2N)}} \approx 1 + {\frac {\varphi'(\tau_n)}{N\varphi(\tau_n)}}.
\end{equation}
Hence the step ratios approach $1$ as $N\rightarrow\infty$, i.e., \emph{locally the method behaves like a constant step size method} for $N$ large enough, since we assumed $\varphi'/\varphi \in L^{\infty}[0,1]$.

It is also of interest to represent the step size change as a relative step size increment, which, in view of (\ref{stepratio}), is defined by
\begin{equation}\label{stepinc}
r_{n-1} = 1 + v_{n-1}\quad \Rightarrow \quad v_{n-1} \approx {\frac {\varphi'(\tau_n)}{N \varphi(\tau_n)}}.
\end{equation}
Thus $v_{n-1}\rightarrow 0$ as $N\rightarrow \infty$, and in practical computations the relative step size increment is invariably small.

The assumption $\varphi'/\varphi \in L^{\infty}[0,1]$ requires that $\log \varphi \in C^1[0,1]$. By a stronger assumption, $\log \varphi \in C^2[0,1]$, we can also estimate the change in the step size ratios,
$$
{\frac {r_n}{r_{n-1}}} = {\frac {h_{n+1}h_{n-1}}{h_n^2}} \approx 1 + {\frac {1}{N^2}}\cdot {\frac {\varphi\varphi'' - (\varphi')^2 }{\varphi^2 }} \approx 1 + {\frac {1}{N^2}}\cdot {\frac {\mathrm d}{\mathrm d\tau}} \left({\frac {\varphi'}{\varphi}}\right) \rightarrow 1,
$$
where $\varphi$ and its derivatives are evaluated at $\tau_n$. Thus the ratio of successive step ratios approach $1$ even faster than the step ratios themselves. The interpretation is that step ratios change slowly, and there may be long strings of consecutive steps where the step size ``ramps up'' as the solution to the ODE gradually becomes smoother after a transient phase. This corresponds to a gradual stretching of the mesh width.
\medskip

\noindent
\textbf{Step sizes and ratios as a function of $t$.} Using $t=\Phi(\tau)$ and $\mathrm d\,t = \varphi \,\mathrm d\tau$, the step size modulation function $\mu(t)$ and the derivative $\varphi(\tau)$ satisfy the functional relation
\begin{equation}\label{funcrel}
\mu(t) = \varphi(\tau).
\end{equation}
Differentiating (\ref{funcrel}) with respect to $t$ and denoting time derivatives by a dot to distinguish them from derivatives with respect to $\tau$, we obtain $\dot \mu \,\mathrm d\,t = \varphi' \, \mathrm d\tau$. Hence
\begin{equation}\label{diffrel}
\dot\mu = \varphi' \cdot {\frac {\mathrm d\tau}{\mathrm d\,t}} = {\frac {\varphi'}{\varphi}},
\end{equation}
allowing us to express step sizes, step ratios, and relative step size increments along the solution of the differential equation, as functions of $t$,
\begin{equation}\label{tdepdata}
h_{n} \approx {\frac {\mu(t_{n+1/2})}{N}}\,;\qquad r_{n-1} \approx 1 + {\frac {\dot\mu(t_n)}{N}}\,;\qquad v_{n-1} \approx {\frac {\dot\mu(t_n)}{N}}.
\end{equation}
Obviously, the previous assumption $\varphi'/\varphi \in L^{\infty}[0,1]$ is equivalent to $\dot\mu \in L^{\infty}[0,1]$. Since $\mu(t) \sim \|y^{(p+1)}(t)\|^{-1/p}$, the assumptions on the deformation map $\Phi$ are realistic and reflect problem regularity.\medskip

\section{Deflation and operator factorization}

The variable step size difference equation
\begin{equation}\label{LMMhom1}
{\frac {1}{h_{n+k-1}}} \sum_{j=0}^k \alpha_{j,n}y_{n+j} = 0,
\end{equation}
can be rewritten in matrix--vector form as
\begin{equation}\label{LMMmatrix}
H_N^{-1}A_{N}(\varphi) y = H_N^{-1}Y_0,
\end{equation}
where the vector $y$ contains all successive approximations $\{y_n\}_{n=1}^{N}$. The vector $Y_0$ is constructed from the initial conditions, $y_0, \dots y_{-k+1}$. Further, $A_{N}(\varphi)$ is an $N\times N$ matrix containing the method coefficients, and is associated with a nonuniform grid characterized by the function $\varphi$. The step sizes are represented by a diagonal matrix $H_N = \tilde\varphi/N$,
\begin{equation}\label{StepMatrix}
H_N = {\mathrm {diag}}(h_0, h_1,\dots h_{N-1}) = {\frac {1}{N}} \, {\mathrm {diag}}(\varphi_{1/2}, \varphi_{3/2},\dots \varphi_{N-1/2}) = {\frac {\tilde\varphi}{N}},
\end{equation}
where $\varphi_{j+1/2} \approx \varphi(\tau_{j+1/2})$. For example, if $k=2$, the matrix $H_N^{-1}A_{N}(\varphi)$ takes the lower tridiagonal form
$$
H_N^{-1}A_{N}(\varphi) = N \tilde\varphi^{-1}
 \begin{pmatrix}
  \alpha_{2,0} & 0 & 0 & 0 & 0\\
  \alpha_{1,1} & \alpha_{2,1} & \cdots & 0 & 0\\
  \alpha_{0,2} & \alpha_{1,2} & \alpha_{2,2} & \cdots & 0 \\
  0 & \alpha_{0,3} & \alpha_{1,3} & \alpha_{2,3} & \cdots \\
 & \ddots  & \ddots  & \ddots & \vdots  \\
  0 & \cdots & \alpha_{0,N-1}  & \alpha_{1,N-1} & \alpha_{2,N-1}
 \end{pmatrix}.
$$
We will investigate zero stability as a question of whether there exists a constant $C_{\varphi}$, independent of $N$, and an $N^*$, such that $\|A_{N}^{-1}(\varphi)H_N\| \leq C_{\varphi}$ for all $N > N^*$. As $\varphi(\tau) \equiv 1$ corresponds to a uniform grid, $A_{N}(1)$ denotes the Toeplitz matrix of method coefficients for constant step size $H_N = I/N$. Then zero stability is equivalent to $\|A_{N}^{-1}(1)/N\| \leq C_1$ for all $N$.

Just as the principal root can be factored out of $\rho$ to construct the extraneous operator $\rho_R(\zeta)$, satisfying $\rho(\zeta) = (\zeta - 1)\rho_R(\zeta)$, a similar factorization holds for the (near) Toeplitz operators. Thus, due to preconsistency ($\rho(1)=0$), the $n^{\mathrm {th}}$ full row sum of the matrix $A_{N}(\varphi)$ is
\begin{equation}\label{zerosum}
\sum_{j=0}^k \alpha_{j,n}(\varphi) = 0,
\end{equation}
even on a nonuniform grid. Denoting the $n^{\mathrm {th}}$ row of nonzeros in $A_{N}(\varphi)$ by a $k+1$ vector $a_n^{\mathrm T}(\varphi)$, and letting $\mathbf {1}_{k+1} = (1\quad 1\dots 1)^{\mathrm T}$ denote a $k+1$ vector of unit elements, preconsistency can be written
\begin{equation}\label{unitorthogonal}
a_n^{\mathrm T}(\varphi)\,\mathbf {1}_{k+1} = 0.
\end{equation}
Hence $a_n^{\mathrm T}(\varphi)$ contains a difference operator. It can therefore be written as a convolution of a $k$-vector $c_n^{\mathrm T}(\varphi)$ and the backward difference operator $\nabla = (-1\quad 1)$, i.e.,
\begin{equation}\label{convolution}
a_n^{\mathrm T}(\varphi) = c_n^{\mathrm T}(\varphi) * \nabla.
\end{equation}
For example, for the constant step size BDF2 method, corresponding to $\alpha_0 = 1/2$, $\alpha_1 = -2$ and $\alpha_2 = 3/2$, the convolution can be represented as
$$
a_n^{\mathrm T}(1) = \left({\frac {1}{2}}\quad -\!2 \quad\:\: {\frac {3}{2}}\right) = \left(0\ \ -{\frac {1}{2}} \quad {\frac {3}{2}}\right) - \left(-{\frac {1}{2}} \quad {\frac {3}{2}} \quad 0 \right),
$$
implying that
$$
c_n^{\mathrm T}(1) = \left(-{\frac {1}{2}} \quad {\frac {3}{2}}\right).
$$
Thus the vector $c_n^{\mathrm T}(1)$ is the $n^{\mathrm {th}}$ full row of nonzero elements of the extraneous operator, corresponding to the coefficients $\gamma_0 = -1/2$ and $\gamma_1 = 3/2$ of the deflated polynomial $\rho_R(\zeta)$. Table \ref{tableBDF} lists the row elements $a_n^{\mathrm T}(1)$ and $c_n^{\mathrm T}(1)$, respectively, for all zero stable BDF methods of step numbers $k\geq 2$.

\begin{table}[tp]
\hspace{30 mm}
\textbf{Coefficients of $A_{k,N}(1)$ and $R_{k,N}(1)$ for BDF2 -- BDF6 methods}
\vspace{-3 mm}

\begin{center}
\begin{tabular}{cccccccc}
\hline\hline\\
\vspace{-3mm}
\textbf{BDF2} \quad $a_n^{\mathrm T}(1)$ & & & & & $1/2$ & $-2$ & $\mathbf{3/2}$ \\
& & & & & & & \\
\hspace{10.5mm} $c_n^{\mathrm T}(1)$ & & & & & & $-1/2$ & $\mathbf{3/2}$ \\
\\
\vspace{-3mm}
\textbf{BDF3} \quad $a_n^{\mathrm T}(1)$ & & & & $-1/3$ & $3/2$ & $-3$ & $\mathbf{11/6}$ \\
& & & & & & & \\
\hspace{10.5mm} $c_n^{\mathrm T}(1)$ & & & & & $1/3$ & $-7/6$ & $\mathbf{11/6}$\\
\\
\vspace{-3mm}
\textbf{BDF4} \quad $a_n^{\mathrm T}(1)$ & & & $1/4$ & $-4/3$ & $3$ & $-4$ & $\mathbf{25/12}$ \\
& & & & & & & \\
\hspace{10.5mm} $c_n^{\mathrm T}(1)$ & & & & $-1/4$ & $13/12$ & $-23/12$ & $\mathbf{25/12}$\\
\\
\vspace{-3mm}
\textbf{BDF5} \quad $a_n^{\mathrm T}(1)$ & & $-1/5$ & $5/4$ & $-10/3$ & $5$ & $-5$ & $\mathbf{137/60}$ \\
& & & & & & & \\
\hspace{10.5mm} $c_n^{\mathrm T}(1)$ & & & $1/5$ & $-21/20$ & $137/60$ & $-163/60$ & $\mathbf{137/60}$\\
\\
\vspace{-3mm}
\textbf{BDF6} \quad $a_n^{\mathrm T}(1)$ & $1/6$ & $-6/5$ & $15/4$ & $-20/3$ & $15/2$ & $-6$ & $\mathbf{147/60}$ \\
& & & & & & & \\
\hspace{10.5mm} $c_n^{\mathrm T}(1)$ & & $-1/6$ & $31/30$ & $-163/60$ & $79/20$ & $-71/20$ & $\mathbf{147/60}$ \\
\\
\hline\hline
\end{tabular}
\caption{Standard constant step size coefficients of BDF$k$ methods denoted by $a_n^{\mathrm T}(1)$, with diagonal elements of the Toeplitz operator $A_{k,N}(1)$ in boldface. The elements appear in each row of $A_{k,N}(1)$, to the left of (below) the diagonal. The corresponding coefficients of the extraneous operator $R_{k,N}(1)$ are denoted by $c_n^{\mathrm T}(1)$, also with diagonal elements in boldface. Note that $c_n^{\mathrm T}(1)\,\mathbf {1}_{k} = 1$.}\label{tableBDF}
\end{center}
\end{table}% This table has been checked and double checked. The coefficients are correct 2017-04-24

Unlike generating polynomials, the (near) Toeplitz operators have the advantage of applying also to nonuniform grids. The following factorization of $H_N^{-1}A_{N}(\varphi)$ is then a matrix representation of the deflation operation described above.

\begin{theorem}\label{FactorizationLemma} Consider a linear multistep method on a nonuniform grid characterized by $\varphi$, and let $H_N$ and $\tilde\varphi$ be defined by (\ref{StepMatrix}). Then $H_N^{-1}A_{N}(\varphi)$ has a factorization
\begin{equation}\label{factorization1}
H_N^{-1}A_{N}(\varphi) = \tilde\varphi^{-1} R_N(\varphi) \cdot \mathbf D_N,
\end{equation}
where $R_N(\varphi)$ is the \emph{extraneous operator}, dependent on the nonuniform grid, and
\begin{equation}\label{factorization2}
\mathbf D_N =  N\begin{pmatrix}
  1 & 0 & 0 & 0 & 0\\
  -1 & 1 & \cdots & 0 & 0\\
  0 & -1 & 1 & \cdots & 0 \\
 & \ddots  & \ddots  & \ddots & \vdots  \\
  0 & \cdots & 0 & -1 & 1
 \end{pmatrix}.
\end{equation}
The \emph{simple integrator} $\mathbf D_N^{-1}$ is stable, and for all $N\geq 1$ it holds that $\|\mathbf D_N^{-1}\|_{\infty} = 1$.
\end{theorem}

\begin{proof} We only need to prove the latter statement. By induction we see that the integrator is a cumulative summation operator,
\begin{equation}\label{inverse}
\mathbf D_N^{-1} = {\frac {1}{N}} \cdot
 \begin{pmatrix}
  1 & 0 & 0 & 0 & 0\\
  1 & 1 & \cdots & 0 & 0\\
 1 & 1 & 1 & \cdots & 0 \\
 & \ddots  & \ddots  & \ddots & \vdots  \\
  1 & \cdots & 1  & 1 & 1
 \end{pmatrix},
\end{equation}
and it immediately follows that $\|\mathbf D_N^{-1}\|_{\infty} = 1$ for all $N$.
\end{proof}

To establish zero stability we need to show that $(H_N^{-1}A_{N}(\varphi))^{-1} = \mathbf D_N^{-1} R_N^{-1}(\varphi)\tilde\varphi$ is uniformly bounded as $N\rightarrow \infty$. We shall use the uniform norm throughout. Since it formally holds that
$$
\|(H_N^{-1}A_{N}(\varphi))^{-1}\|_{\infty} \leq \|\mathbf D_N^{-1}\|_{\infty}\cdot \|R_N^{-1}(\varphi)\|_{\infty}\cdot\|\tilde\varphi\|_{\infty},
$$
where $\|\tilde\varphi\|_{\infty}$ is bounded for all smooth grids, the remaining difficulty is to show that $\|R_N^{-1}(\varphi)\|_{\infty}\leq C_{\varphi}$ for all $N>N^*$, and how this depends on grid regularity. For a unform grid, zero stability is determined by the roots of the extraneous operator; this needs to be translated into norm conditions. A simple possibility is to use the fact that
$$
m_{\infty}[R_N(1)] > 0 \quad \Rightarrow \quad \|R_{N}^{-1}(1)\|_{\infty} \leq {\frac {1}{m_{\infty}[R_N(1)]}},
$$
where $m_{\infty}[R_N(1)]$ is the lower logarithmic norm of $R_N(1)$, see \cite{Soderl2006}. The condition $m_{\infty}[R_N(1)] > 0$ is equivalent to \textit{diagonal dominance}. For example, by Table \ref{tableBDF}, the BDF2 matrix $NA_{2,N}(1)$ associated with the $\rho$ operator has the factorization
\begin{equation}\label{extraBDF2}
NA_{2,N}(1) =
 \begin{pmatrix}
  3/2 & 0 & 0 & 0 & 0\\
  -1/2 & 3/2 & \cdots & 0 & 0\\
 0 & -1/2 & 3/2 & \cdots & 0 \\
 & \ddots  & \ddots  & \ddots & \vdots  \\
  0 & \cdots & 0  & -1/2 & 3/2
 \end{pmatrix} \cdot \mathbf D_N = R_{2,N}(1) \cdot \mathbf D_N.
\end{equation}
where the nonzero coefficients correspond to the $c_n^{\mathrm T}(1)$ vector of Table \ref{tableBDF}. Since
\begin{equation}\label{mP2}
m_{\infty}[R_{2,N}(1)] = {\frac {3}{2}}-{\frac {1}{2}} = 1 > 0,
\end{equation}
it follows that $\| R_{2,N}^{-1}(1) \|_{\infty} \leq 1/m_{\infty}[ R_{2,N}(1) ] = 1$ and that the BDF2 method is zero stable. The same technique works for the BDF3 method, since
$$
m_{\infty}[R_{3,N}(1)] = {\frac {11}{6}}-{\frac {7}{6}} - {\frac {1}{3}} = {\frac {1}{3}} > 0.
$$
However, it fails for the BDF4 method and higher, since the extraneous operator is then no longer diagonally dominant. By instead computing e.g.\ the Euclidean norm numerically, the above technique can be extended to BDF4 and BDF5, but it again fails for BDF6. For this reason, we need a general result, based on sharper estimates.

\begin{theorem}\label{T0inv}
For every strongly stable $k$-step method on a uniform grid, there is a constant $C_0 < \infty$, such that $\|R_{k,N}^{-1}(1)\|_{\infty} \leq C_0$ for all $N\geq 1$.
\end{theorem}

\begin{proof}
Let $T_0$ denote the lower triangular Toeplitz matrix $R_{k,N}(1)$ representing the extraneous operator. Then $T_0^{-1}$ is lower triangular too, albeit full. More importantly, $T_0^{-1}$ is also Toeplitz. By (\ref{rhooperators}), $\rho_R(\zeta) = \sum_{j=0}^{k-1} \gamma_j\zeta^j$. Noting that $\alpha_k = \gamma_{k-1}$, and illustrating the matrix $T_0$ for $k=3$, we have
$$
T_0 =
\begin{pmatrix}
  \gamma_{2} & 0 & \cdots & 0 & 0\\
  \gamma_{1} & \gamma_{2} & \cdots & \cdots & 0\\
  \gamma_{0} & \gamma_{1} & \gamma_{2} & \cdots & \\
  0 & \ddots  & \ddots  & \ddots & \vdots  \\
  \dots & 0 & \gamma_{0} & \gamma_{1} & \gamma_{2}
 \end{pmatrix}=
\alpha_{3} \begin{pmatrix}
  1 & 0 & \cdots & 0 & 0\\
  \delta_{1} & 1 & \cdots & \cdots & 0\\
  \delta_{0} & \delta_{1} & 1 & \cdots & \\
  0 & \ddots  & \ddots  & \ddots & \vdots  \\
  \cdots & 0 & \delta_{0} & \delta_{1} & 1
 \end{pmatrix} = \alpha_{3} \hat T_0,
$$
where, in the general case, $\delta_j = \gamma_j/\alpha_{k}$ are the elements of the scaled matrix $\hat T_0$, with Toeplitz inverse
$$
\hat T_0^{-1} = \begin{pmatrix}
  1 & 0 & \dots & 0 & 0\\
  u_2 & 1 & \cdots & \cdots & 0\\
  u_3 & u_2 & 1 & \cdots & \\
  \vdots & \ddots  & \ddots  & \ddots & \vdots  \\
  u_N & \dots & u_3 & u_2 & 1
 \end{pmatrix}.
$$
Hence $\|\hat T_0^{-1}\|_{\infty} \leq C$ as $N\rightarrow \infty$ if and only if the sequence $u = \{u_n\}_{n=1}^{N}$ (where we define $u_1=1$) is in $l^1$, i.e., the sequence $u$ must be absolute summable as $N\rightarrow \infty$. By construction, $u$ satisfies the difference equation $\rho_R(\mathrm E) u = 0$, where $\mathrm E$ is the forward shift operator. By assumption $\rho_R(\zeta)$ satisfies the strict root condition. Therefore $u$ is bounded, i.e., $u\in l^{\infty}$. Let $\rho_R(\zeta_{\nu})=0$ and let
$$
\max_{\nu} |\zeta_{\nu}| \leq q < 1,
$$
where equality applies whenever the maximum modulus root is simple. Then there is a constant $K < \infty$ such that $|u_n| \leq K \cdot q^n$ for all $n\geq 1$. Hence $u\in l^{1}$, as
$$
\|u\|_1 = \sum_1^N |u_n| \leq K \sum_1^{\infty} q^n = {\frac {Kq}{1-q}}.
$$
Since $\|\hat T_0^{-1}\|_{\infty} = \|u\|_1$ due to the Toeplitz structure of $\hat T_0^{-1}$, we have, for all $N\geq 1$,
$$
\|R_{k,N}^{-1}(1)\|_{\infty} = \|T_0^{-1}\|_{\infty} \leq {\frac {Kq}{(1-q)\cdot \alpha_{k}}} \leq C_0,
$$
and the proof is complete.
\end{proof}

\section{Zero stability on nonuniform grids -- the BDF2 method}

The general proof of variable step size zero stability is based on the operator factorization given by Theorem \ref{FactorizationLemma}. Beginning with an example, the variable step size BDF2 discretization of $\dot y = 0$ is
\begin{equation}\label{bdf2full}
{\frac {1}{2h_{n+1}}} (r_n^2y_n - (1 + r_n)^2y_{n+1} + (1 + 2r_n) y_{n+2}) = 0,
\end{equation}
where $r_n = h_{n+1}/h_n$ is the step ratio. Rearranging terms, we obtain
\begin{equation}\label{bdf2fac}
{\frac {1}{2h_{n+1}}} \left( [-r_n^2y_{n+1} + (1 + 2r_n)y_{n+2}] - [-r_n^2y_{n} + (1 + 2r_n)y_{n+1}] \right) = 0.
\end{equation}
Using $h_{n+1}=\varphi_{n+1/2}/N$, we can factor out the simple integrator to obtain
\begin{equation}\label{bdf2extrafac}
-{\frac {r_n^2}{2\varphi_{n+1/2}}}\,{\frac {y_{n+1}-y_n}{1/N}} + {\frac {1+2r_n}{2\varphi_{n+1/2}}}\,{\frac {y_{n+2}-y_{n+1}}{1/N}} = 0.
\end{equation}
Introducing $u_n/N = y_{n+1}-y_n$, the ``extraneous recursion'' becomes
\begin{equation}\label{bdf2extrafac2}
-{\frac {r_n^2}{2\varphi_{n+1/2}}}\,u_n + {\frac {1+2r_n}{2\varphi_{n+1/2}}}\,u_{n+1} = 0.
\end{equation}
As the subsequent Euler integration $y_{n+1} = y_n + u_n/N$ is stable (cf.\ Theorem \ref{FactorizationLemma}), the composite scheme is stable provided that the one-step recursion (\ref{bdf2extrafac2}) is stable. Obviously, $|u_{n+1}| \leq |u_n|$ provided that
$$
{\frac {r_n^2}{1+2r_n}} \leq 1,
$$
which holds for $0 < r_n \leq 1 + \sqrt{2}$. This bound on the step ratio is the same as the classical bound found in \cite[p.\ 405--406]{HairNW1993}.

In terms of the (near) Toeplitz operators used above, the variable step size extraneous operator is given by
$$
R_{2,N}(\varphi) =
 {\frac {1}{2}}\begin{pmatrix}
  1+2r_1 & 0 & 0 & 0 & 0\\
  -r_2^2 & 1+2r_2 & \cdots & 0 & 0\\
  & -r_3^2 & 1+2r_3 & \cdots & 0 \\
 & \ddots  & \ddots  & \ddots & \vdots  \\
  0 & 0 &  & -r_N^2 & 1+2r_N
 \end{pmatrix}.
$$
The operator $R^{-1}_{2,N}(\varphi)$ is bounded whenever the lower logarithmic max norm,
\begin{equation}\label{lowermR2}
m_{\infty}[R_{2,N}(\varphi)] = \min (1 + 2r - r^2) > 0
\end{equation}
along the range of step ratios $r$. Diagonal dominance requires that $1+2r-r^2 > 0$, which holds if $0 < r < 1 + \sqrt{2}$, so the classical bound is obtained once more. As we assume a smooth grid in terms of (\ref{stepratio}), with $\dot\mu = \varphi'/\varphi \in L^{\infty}[0,1]$, the condition $r_n < 1 + \sqrt{2}$ is fulfilled for
\begin{equation}\label{minimumN}
N > N^* = {\frac {\|\varphi'/\varphi\|_{\infty}}{\sqrt{2}}} = {\frac {\|\dot\mu\|_{\infty}}{\sqrt{2}}}.
\end{equation}
%%%%  -----------------------

In general, however, a method can be zero stable without diagonal dominance, requiring more elaborate techniques to establish zero stability. The variable step size discretization (\ref{LMMhom1}) of $\dot y=0$ is factorized to obtain the difference equation corresponding to the extraneous operator only,
\begin{equation}\label{LMMhomG1}
\sum_{j=0}^{k-1} \gamma_{j,n}u_{n+j} = 0,
\end{equation}
where the coefficients $\gamma_{j,n}$ are multivariate rational functions of $k-1$ consecutive step size ratios. If the sequence $u$ is bounded (zero stability), then the original solution $y$ of (\ref{LMMhom1}) is obtained by simple Euler integration, $y_{n+1} = y_n + u_n/N$, where $h=1/N$ is a constant step size and $N\rightarrow \infty$. Since the latter integration is stable, we only need to bound the solutions $u$ of (\ref{LMMhomG1}). Using (\ref{stepinc}), we write the step ratios
$$
r_n = 1 + v_n,
$$
where, for smooth grids,
$$
|v_n| \leq {\frac {1}{N}} \left\| {\frac {\varphi'}{\varphi}} \right\|_{\infty}.
$$
Thus, the larger the value of $N$, the closer is $|v_n|$ to zero. Now, for $v_n\equiv 0$ we obtain the classical constant step size method. The difference equation (\ref{LMMhomG1}) can then be rearranged as a Toeplitz system $T_0 u = U_0$, where $T_0 = R_{2,N}(1)$ and $u = \{u_n\}_1^N$ denotes the entire solution. The vector $U_0$ contains initial data as needed. By Theorem \ref{T0inv}, we have $\|T_0^{-1}\|_{\infty} \leq C_0$ for all $N\geq 1$.

With variable steps, the system will depend on the step ratios, and the overall system matrix will no longer be Toeplitz. Nevertheless, for the BDF2 example used above, we have seen that the extraneous system matrix can be written
\begin{align*}
R_{2,N}(\varphi) &=
 {\frac {1}{2}}\begin{pmatrix}
  3+2v_1 & 0 & 0 & 0 & 0\\
  -1 - 2v_2 - v_2^2 & 3+2v_2 & \cdots & 0 & 0\\
  & -1 - 2v_3 - v_3^2 & 3+2v_3 & \cdots & 0 \\
 & \ddots  & \ddots  & \ddots & \vdots  \\
  0 & 0 &  & -1 -2v_N - v_N^2 & 3+2v_N
 \end{pmatrix}\\
 &=
 T_0 + V\begin{pmatrix}
  1 & 0 & 0 & 0 & 0\\
  -1 & 1 & \cdots & 0 & 0\\
  & -1 & 1 & \cdots & 0 \\
 & \ddots  & \ddots  & \ddots & \vdots  \\
  0 & 0 &  & -1 & 1
 \end{pmatrix} + {\frac {V^2}{2}}\begin{pmatrix}
  0 & 0 & 0 & 0 & 0\\
  -1 & 0 & \cdots & 0 & 0\\
  & -1 & 0 & \cdots & 0 \\
 & \ddots  & \ddots  & \ddots & \vdots  \\
  0 & 0 &  & -1 & 0
 \end{pmatrix}.
\end{align*}
Thus we can write
\begin{equation}\label{powerseries}
R_{2,N}(\varphi) = T_0 + VT_1 + V^2T_2 = (I + VT_1T_0^{-1} + V^2T_2T_0^{-1})T_0,
\end{equation}
where the $T_j$ are Toeplitz and $V = \mathrm {diag}(v_j)$ is a diagonal matrix. Since $T_0^{-1}$ is uniformly bounded, a sufficient condition for $R_{2,N}(\varphi)$ to be invertible is
\begin{equation}\label{nonlincond}
\|V\|_{\infty}\cdot\|T_1T_0^{-1}\|_{\infty} + \|V\|_{\infty}^2\cdot\|T_2T_0^{-1}\|_{\infty} < 1,
\end{equation}
and we obtain the bound
\begin{equation}\label{invboundBDF2}
\|R_{2,N}(\varphi)^{-1}\|_{\infty} \leq {\frac {\|T_0^{-1}\|_{\infty}}{1-\|V\|_{\infty}\cdot\|T_1T_0^{-1}\|_{\infty}-\|V\|_{\infty}^2\cdot\|T_2T_0^{-1}\|_{\infty}}}.
\end{equation}
Here the $\|T_jT_0^{-1}\|_{\infty}$ are method dependent constants, and
\begin{equation}\label{Vbound}
\|V\|_{\infty} = {\frac {1}{N}} \left\| {\frac {\varphi'}{\varphi}} \right\|_{\infty}.
\end{equation}
We can now determine a sufficient condition on $\|V\|_{\infty}$ in general, and on $N$ in particular, such that (\ref{nonlincond}) is satisfied. Because $w := \|V\|_{\infty} = \mathrm O(N^{-1})$ if the grid is regular, there is always an $N$ large enough to satisfy this condition. Considering the equation
\begin{equation}\label{condzero}
w \cdot\|T_1T_0^{-1}\|_{\infty} + w^2\cdot\|T_2T_0^{-1}\|_{\infty} = 1,
\end{equation}
we find that we have to take $N$ large enough to guarantee that
$$
{\frac {1}{N}} \left\| {\frac {\varphi'}{\varphi}} \right\|_{\infty} < {\frac {-\|T_1T_0^{-1}\|_{\infty} + \sqrt{\|T_1T_0^{-1}\|_{\infty}^2 + 4\|T_2T_0^{-1}\|_{\infty}}}{2\|T_2T_0^{-1}\|_{\infty}}}.
$$
The quantity on the right hand side depends only on the method coefficients, and the left hand side depends only on the total number of steps, and the regularity of the nonuniform grid, as measured by $\|\varphi'/\varphi\|_{\infty}$.

\section{Zero stability on nonuniform grids -- Higher order methods}

%%=== THIS PARAGRAPH IS REWORKED ===%%
In a $k$-step method using variable steps, the coefficients depend on $k-1$ step ratios. This makes the problem significantly more difficult. Without loss of generality, we will only consider an approach linear in $V$ below. Note that while $\|V\|_{\infty} = \mathrm O(N^{-1})$, it follows that higher powers of $V$ are $\|V\|_{\infty}^k = \mathrm O(N^{-k})$, implying that they are significantly smaller than the first order term when $N$ is large and the grid is smooth. For example, in (\ref{condzero}) above, we have $w=\mathrm O(N^{-1})$ implying that the $w^2$ is negligible as $N\rightarrow \infty$; it is therefore sufficient to consider terms of order $\mathrm O(N^{-1})$ only. This overcomes the added difficulty of considering $k$-step methods.

The procedure for a general $k$-step method follows the same pattern as the in the previous examples. Neglecting quadratic and higher order terms in $V$, the extraneous operator is
\begin{equation}\label{general}
R_{k,N}(\varphi) = T_0 + \sum_{j=1}^{k-1} V_jT_j  = \left( I + \sum_{j=1}^{k-1} V_jT_jT_0^{-1} \right) T_0 + \mathrm O(N^{-2}).
\end{equation}
The diagonal matrices $V_j$ only differ in the diagonal elements being successively shifted down the diagonal. Assume that $\log\varphi \in C^2[0,1]$. By (\ref{stepinc}) and the mean value theorem,
$$
v_{n+1} - v_n = {\frac {1}{N}} \left( {\frac {\varphi'(\tau_{n+1})}{\varphi(\tau_{n+1})}} - {\frac {\varphi'(\tau_{n})}{\varphi(\tau_{n})}} \right) \approx 1 + {\frac {1}{N^2}}\cdot {\frac {\varphi\varphi'' - (\varphi')^2 }{\varphi^2 }},
$$
evaluating $\varphi$ and its derivatives at $\tau_{n+1/2}$. It follows that $V_{j+1} = V_j + \mathrm O(N^{-2})$, and that all $V_j$ can be replaced by a single matrix, $V$, while only incurring $\mathrm O(N^{-2})$ perturbations. Further, (\ref{Vbound}) holds for all $V_j$.

Since $\|T_0^{-1}\|_{\infty} \leq C_0$, a sufficient condition for the extraneous operator $R_{k,N}(\varphi)$ to have a uniformly bounded inverse is
\begin{equation}\label{final}
{\frac {1}{N}} \left\| {\frac {\varphi'}{\varphi}} \right\|_{\infty} \cdot \sum_{j=1}^{k-1} \|T_jT_0^{-1}\|_{\infty} \: < \: 1.
\end{equation}
This condition separates grid smoothness $\|\varphi'/\varphi\|_{\infty}$ from method parameters, as represented by the Toeplitz matrices $T_j$. Thus, in order to prove zero stability as $N\rightarrow \infty$, we need $\|T_j\| \leq C_j$ for $j\geq 1$. The latter condition is easily established, once the coefficients' dependence on the step ratios has been established. Hence we have the following general result.

\begin{theorem}\label{mainresult}
For all smooth maps $\Phi$ there exist constants $N^*$ and $C_{\varphi}$ (independent of $N$) such that $\|R_{k,N}^{-1}(\varphi)\|_{\infty} \leq C_{\varphi}$ for $N > N^*$, whenever $\|R_{k,N}^{-1}(1)\|_{\infty} \leq C_0$ for all $N$.
\end{theorem}

To illustrate the general theory, we consider the variable step size BDF3 method. Slightly modifying the conventions set out in Section 2, we define
\begin{equation}\label{conventions}
h_{n-1}=t_{n}-t_{n-1}, \quad r_1=\frac{h_{n-1}}{h_{n-2}}, \quad r_2=\frac{h_{n-2}}{h_{n-3}},
\end{equation}
where $r_1=1+v_1$ and $r_2=1+v_2$ denote the step ratios that will occur in a single row of the Toeplitz operator. Naturally, these values change from one row to the next, as they depend on $n$ as indicated by (\ref{conventions}). Within this setting, after deflating the operator, we obtain a recursion on a nonuniform grid corresponding to
$$
\gamma_0(r_1,r_2)z_{n-2} + \gamma_1(r_1,r_2)z_{n-1} + \gamma_2(r_1,r_2)z_n = 0,
$$
where
\begin{align*}
\gamma_2(r_2,r_1) &= \frac{4r_1r_2+r_2+3r_1^2r_2+2r_1+1}{r_2+r_1^2r_2+2r_1r_2+r_1+1} \\
\gamma_1(r_2,r_1) &= -\frac{r_1^2(4r_1r_2^2+r_1^2r_2^2+1+2r_1r_2+3r_2+3r_2^2)}{(r_2+r_1^2r_2+2r_1r_2+r_1+1)(r_2+1)} \\
\gamma_0(r_2,r_1) &= \frac{(r_1+1)r_1^2r_2^3}{(r_2+1)(r_1r_2+r_1+1)}.
\end{align*}
The coefficients are normalized so that $\beta_{3,n}=1$. (In a general analysis, they are normalized by $\beta_{k,n}=1$, cf.\ (\ref{LMMvar})). By writing $r_j=1+v_j$, where $v_j = \mathrm O(N^{-1})$, we obtain
\begin{eqnarray*}
\gamma_2(v_2,v_1) &=& \frac{11+8v_2+12v_1+10v_1v_2+3v_1^2+3v_1^2v_2}{6+4v_2+5v_1+4v_1v_2+v_1^2+v_1^2v_2}\\
\gamma_1(v_2,v_1) &=& -\frac{(1+v_1)^2(14+21v_2+8v_2^2+8v_1+14v_1v_2+6v_1v_2^2+v_1^2+2v_1^2v_2+v_1^2v_2^2)}{(6+4v_2+5v_1+4v_1v_2+v_1^2+v_1^2v_2)(2+v_2)}\\
\gamma_0(v_2,v_1) &=& \frac{(2+v_1)(1+v_1)^2(1+v_2)^3}{(2+v_2)(3+2v_2+v_1+v_1v_2)}.
\end{eqnarray*}
Since $v_j = \mathrm O(N^{-1})$ we drop higher order terms to obtain
\begin{eqnarray*}
\gamma_2(v_2,v_1) &\approx & {\frac {11 + 12v_1 + 8v_2}{6 + 5v_1 + 4v_2}} \approx \frac{66+17v_1+4v_2}{36}\\
\gamma_1(v_2,v_1) &\approx & -{\frac {14 + 36v_1 + 21v_2}{12 + 16v_1 + 8v_2}} \approx -\frac{42+73v_1+14v_2}{36}\\
\gamma_0(v_2,v_1) &\approx & {\frac {2 + 5v_1 + 6v_2}{6 + 2v_1 + 7v_2}} \approx \frac{12+26v_1+22v_2}{36}.
\end{eqnarray*}
We can now identify three lower triangular Toeplitz operators, with diagonal elements in boldface,
\begin{eqnarray*}
T_0 &=&  \frac {1}{36} \: [\: 12  \quad -42 \quad \mathbf{66} \: ]\\
T_1 &=& \frac {1}{36} \: [\: 26 \quad -73 \quad \mathbf{17} \: ]\\
T_2 &=& \frac {1}{36} \: [\: 22 \quad -14 \quad \mathbf{4} \: ].
\end{eqnarray*}
These correspond to the $T_j$ matrices in (\ref{final}), and the matrices $V_1$ and $V_2$ are just diagonal matrices collecting the sequences of $v_1$ and $v_2$ values along the grid.

Because $v_2-v_1 = \mathrm O(N^{-2})$, we may consider a further simplification, putting $v_2 = v_1$, or, equivalently, $r_2 = r_1$. This corresponds to ``ramping up'' the step size at an exponential rate, and is particularly challenging to zero stability. In such a case, we may consider $T_0 + V(T_1+T_2)$, with elements rescaled to have a common denominator,
\begin{eqnarray*}
T_0 &=&  \frac {1}{12} \: [\: 4  \quad\:\:\: -14 \quad \mathbf{22} \: ]\\
(T_1 + T_2)v &=& \frac {v}{12} \: [\: 16  \quad -29 \quad\:\:\: \mathbf{7} \: ].
\end{eqnarray*}
Here the diagonal dominance of $T_0$ is sufficient to derive a condition for zero stability. We can thus compute the lower logarithmic max norm,
$$
m_{\infty}[T_0 + (T_1+T_2)v] = {\frac {22+7v - |14+29v| - |4+16v|}{12}} = {\frac {2 - 19v}{6}}.
$$
where we have assumed that $v > -1/4$, allowing the removal of absolute values. Thus $m_{\infty}[T_0 + (T_1+T_2)v] > 0$ if $v < 2/19$. By requiring
$$
\|V\|_{\infty} = {\frac {1}{N}} \left\|{\frac {\varphi'}{\varphi}}\right\|_{\infty} < {\frac {2}{19}}
$$
the operator $T_0 + V\cdot (T_1+T_2)$ has a uniformly bounded inverse. The corresponding zero stability condition is
$$
N > N^* = {\frac {19}{2}} \left\| {\frac {\varphi'}{\varphi}} \right\|_{\infty}.
$$
For BDF3 \cite[p.\ 406]{HairNW1993} cite Grigorieff's (1983) \emph{sufficient} conditions for zero stability,
$$
0.836 < {\frac {h_k}{h_{k-1}}} < 1.127.
$$
Our BDF3 bounds for ramp-up provide the conditions
$$
0.75 = 1 - {\frac {1}{4}} < {\frac {h_n}{h_{n-1}}} < 1 + {\frac {2}{19}} \approx 1.105.
$$
The differences between these results depend on the methodology, and not least on the choice of norm. The deflation approach used here is similar to the technique used in \cite{Grigor1983}, while smooth grid maps are akin to the assumptions used in \cite{GearTu1974}.

It is important to note that we do not try to determine the greatest possible step size increase, but instead prove that every strongly stable method will be zero stable on smooth grids. We have also seen that the complexity of determining exact stability bounds quickly becomes overwhelming, which is why we argue that an alternative proof, revealing the dependence on smoothness and method parameters, is sufficient.

\section{Conclusions}

%%=== THIS SECTION IS REVISED TO CLARIFY SOME ISSUES RAISED BY ONE OF THE REFEREES ===%%
In this paper we have demonstrated that any linear multistep method which is strongly stable on a uniform grid is also zero stable on any smooth nonuniform grid. Grid smoothness is (in theory) determined by a grid map $\Phi:[0,1] \rightarrow [0,1]$, satisfying $\Phi(0)=0$ and $\Phi(1)=1$, and having a strictly positive derivative $\varphi = \Phi'$. The grid map transforms a uniform grid of $N$ steps into a nonuniform grid, which is smooth if $\log \varphi$ is continuously differentiable.

In practice, this corresponds to a smooth step size variation, where the step size at time $t\in[0,1]$ can be represented by a continuous modulation function, so that $h(t) = \mu(t)/N$. Here $\dot\mu(t) = \varphi'/\varphi$, which must remain bounded. The modulation function $\mu(t)$ is determined by the solution of the differential equation, while $N$ is determined by the accuracy requirement as specified by the tolerance $\varepsilon$.

The main result is that every $k$-step method is associated with $k$ bounded Toeplitz operators $T_0,\dots T_{k-1}$, where $T_0$ is associated with the constant step size method. If that method is strongly zero stable, then $T_0$ has a bounded inverse. Smooth step size variation is characterized locally by the function $\varphi'/\varphi$, the magnitude of which determines how many steps $N$ that need to be taken in order to guarantee variable step size zero stability. Thus, if
$$
{\frac {1}{N}} \left\| {\frac {\varphi'}{\varphi}} \right\|_{\infty} \cdot \sum_{j=1}^{k-1} \|T_jT_0^{-1}\|_{\infty} \: < \: 1.
$$
the numerical solution to $\dot y = 0$ is stable. Examples are given for BDF methods.

This result is also practically significant as it implies that time step adaptivity must be implemented using smooth step size changes, such that consecutive step ratios are $r = 1 + \mathrm O(h)$. This can easily be achieved, as there is a wide range of smooth controllers available for dedicated purposes, \cite{Soderl2003}. These are based on digital filter theory, and control $\log h$ in small increments, changing the step size on every step. Since $h \sim \varphi/N$, such a controller keeps $\log \varphi$ smooth, in line with the assumptions of Theorem \ref{mainresult}. The smoothness requirement is local, and does not imply any bound on $h_{\mathrm {max}}/h_{\mathrm {min}}$. It is therefore not a limitation in stiff computation, where overall step size variation necessarily is large.

\begin{acknowledgements}
The authors gratefully acknowledge the contribution of Prof.\ Carmen Ar\'evalo, who provided the grid-independent variable step size coefficients for the BDF3 method, computed in \textsc{Maple}.
\end{acknowledgements}

\end{document}